\documentclass[a4paper,12pt]{amsart}
\usepackage{amsfonts}
\usepackage{mathrsfs}
\usepackage{amsmath,amssymb,latexsym,amsfonts,amscd}
\usepackage{showlabels}

\title{An optimal boundedness on weak $\bQ$-Fano threefolds}
\author{Jungkai A. Chen and Meng Chen}
\address{\rm Taida Institute for Mathematical Sciences,
National Center for Theoretical Sciences, Taipei Office, and
Department of Mathematics, National Taiwan University, Taipei, 106,
Taiwan} \email{jkchen@math.ntu.edu.tw}

\address{\rm Institute of Mathematics, Fudan University,
Shanghai, 200433, People's Republic of China}
\email{mchen@fudan.edu.cn}

\thanks{The first author was partially supported by TIMS, NCTS/TPE
and  National Science Council of Taiwan. The second author was
supported by both the Program for New Century Excellent Talents in
University (\#NCET-05-0358) and the National Outstanding Young
Scientist Foundation (\#10625103)}

\newcommand{\bQ}{{\mathbb Q}}

\newcommand{\rounddown}[1]{\lfloor{#1}\rfloor}

\newcommand\OO{{\mathcal{O}}}

\newcommand\tp{{\tilde{P}}}
\newcommand\mB{{B}}

\newtheorem{thm}{Theorem}[section]
\newtheorem{lem}[thm]{Lemma}

\newtheorem{prop}[thm]{Proposition}

\theoremstyle{definition}

\newtheorem{setup}[thm]{}

\newtheorem{exmp}[thm]{Example}

\theoremstyle{remark}

\newtheorem{CLM}{Claim}

\begin{document}
\begin{abstract}
Let $X$ be a terminal weak $\bQ$-Fano threefold. We prove that
$P_{-6}(X)>0$ and $P_{-8}(X)>1$. We also prove that the
anti-canonical volume has a universal lower bound $-K_X^3 \geq
1/330$. This lower bound is optimal.
\end{abstract}
\maketitle

\section{\bf Introduction}

A threefold $X$ is said to be a  terminal (resp. canonical)
$\bQ$-Fano threefold if $X$ has at worst terminal (resp. canonical)
singularities and  $-K_X$ is ample, where $K_X$ is a canonical Weil
divisor on $X$.
 $X$ is called a {\it terminal weak $\bQ$-Fano threefold} if $X$
has at worst terminal  singularities and  $-K_X$ is nef and big.

We are interested in a conjecture of Miles Reid \cite[Section
4.3]{YPG} which says that $P_{-2}(X)>0$ for almost all $\bQ$-Fano
3-folds. There are already several known
examples with $P_{-2}=0$ by Iano-Fletcher \cite{Fletcher} and
Altinok and Reid \cite{AR}. Another question that we are interested
in is the boundedness of $\bQ$-Fano 3-folds, which is equivalent to
the boundedness of the anti-canonical volume $-K^3_X$. Kawamata
\cite{KA} first showed the boundedness of $-K^3$ for terminal
$\bQ$-Fano 3-folds with Picard number 1. Then Koll\'ar, Miyaoka,
Mori and Takagi \cite{KMMT} gave the boundedness for all canonical
$\bQ$-Fano 3-folds. Recently Brown and Susuki \cite{BS} proved a
sharp lower bound of $-K^3$ for certain $\bQ$-Fano 3-folds. However
a practical lower bound of $-K^3$ for all $\bQ$-Fano 3-folds is
still unknown, which is another motivation of our
paper.

Our main results are the following:

\begin{thm}\label{main} Let $X$ be a terminal weak $\bQ$-Fano
3-fold. Then
\begin{itemize}
\item[(i)] $P_{-4}>0$ with possibly one exception of basket of singularities;
\item[(ii)] $P_{-6}>0$ and $P_{-8}>1$;
\item[(iii)] $-K_X^3\geq \frac{1}{330}$. Furthermore
$-K^3_X=\frac{1}{330}$ if and only if the virtual basket of
singularities is
$$\{\frac{1}{2}(1,-1,1), \frac{1}{5}(1,-1,2),
\frac{1}{3}(1,-1,1), \frac{1}{11}(1,-1,2)\}.$$
\end{itemize}
\end{thm}

The lower bound $\frac{1}{330}$ is optimal due to the following
example by Iano-Fletcher:

\begin{exmp}(\cite[Page 158]{Fletcher}) The general hypersurface
$$X_{66} \subset \mathbb{P}(1,5,6,22,33)$$ has
$-K_X^3=\frac{1}{330}$.
\end{exmp}

We now sketch our method of baskets and explain the idea of the
proofs. Recall that Reid's Riemann-Roch formula describes the Euler
characteristic by counting the contribution from virtual quotient
singularities, which he calls {\it basket}. We remark that when
either $K_X$ or $-K_X$ is nef and big, then Euler characteristic is
nothing but plurigenus or anti-plurigenus. Our method in
\cite{explicit} provides a synthetic way to recover baskets in
terms of plurigenera. Even though one can not expect to recover
baskets completely with limited information from plurigenera.
However the possibility of baskets is finite when $P_{-m}$ is small
for small $m$.

The behavior of baskets in $\bQ$-Fano case is somehow nicer. One
reason is that $\chi(\OO_X)=1$. And thanks to many effective
inequalities derived from the basket trick, we can prove that there
are only finitely many baskets with given $P_{-1}$ and $P_{-2}$ (see
\ref{finite}). Furthermore we can give a complete list of those
small anti-plurigenera formal baskets satisfying geometric
constrains (2.1),(2.2) and (2.3). This allows us to prove our
statements.
\bigskip

We would like to thank J\'anos Koll\'ar for his comment on the
possibility of using our basket consideration in \cite{explicit} to
classify $\bQ$-Fano 3-folds with small anti-volume. We are grateful
to Miles Reid for telling us their interesting examples with
$P_{-2}=0$. Thanks are also due to De-Qi Zhang for the effective
discussions during the preparation of this note.

\section{\bf Baskets of pairs and geometric inequalities}
In this section, we would like to recall our method developed in
\cite{explicit}, together with some geometrical inequalities which
will be the core of our proof.

A {\it basket} $B$  is a collection of pairs of integers (permitting
weights) $\{(b_i,r_i)|i=1,\cdots,t; b_i\ \text{coprime
 to}\ r_i\}$.\footnote{We may drop the assumption of coprime if we simply consider $\{(db,dr),*\}$ as $\{d \times (r,b),*\}$. These two baskets
 share all the same numerical properties in our discussion.} For simplicity, we will frequently write a basket in another way,
 say
$$B=\{(1,2), (1,2), (2,5)\}=\{2\times (1,2), (2,5)\}.$$

\begin{setup}\label{RR}{\bf Reid's formula.}
Let $X$ be a terminal weak $\bQ$-Fano 3-fold. According to Reid
\cite{YPG}, there is a basket of pairs
$$B_X=\{(b_i,r_i)|i=1,\cdots, t; 0<b_i\leq \frac{r_i}{2};b_i \text{ is coprime to } r_i\}.$$
such that, for all integer $n>0$,
$$P_{-n}(X)=\frac{1}{12}n(n+1)(2n+1)(-K_X^3)+(2n+1)-l(-n)   $$
where
$l(-n)=l(n+1)=\sum_i\sum_{j=1}^n\frac{\overline{jb_i}(r_i-\overline{jb_i})}{2r_i}$
and $\overline{\cdot}$ means the smallest residue mod $r_i$.

The above formula can be rewritten as:

\begin{eqnarray*}
{P}_{-1}&=&
\frac{1}{2}(-K_X^3+\sum_i \frac{b_i^2}{r_i})-\frac{1}{2}\sum_i b_i+3,\\
{P}_{-m}-{P}_{-(m-1)}&=& \frac{m^2}{2}(-K_X^3+\sum_i
\frac{b_i^2}{r_i})-\frac{m}{2}\sum_i b_i+2-\Delta^{m}
\end{eqnarray*}
where $\Delta^{m}= \sum_i
(\frac{\overline{b_im}(r_i-\overline{b_im})}{2r_i}-
\frac{b_im(r_i-b_im)}{2r_i})$ and $m\geq 2$.
\medskip

Notice that all the anti-plurigenera  $P_{-n}$ can be determined by
the basket $B_X$ and $P_{-1}(X)$. This leads us to set the
following definitions for {\it formal baskets}.
\end{setup}

We recall some definitions and properties of baskets. Especially, we
introduce the notion of packing. All details can be found in Section
4 of \cite{explicit}.

Suppose that $B=\{(b_i,r_i)|i=1,\cdots, t; 0<b_i\leq \frac{r_i}{2}; b_i
\text{\ is coprime to\ } r_i\}$ is a basket. Let $n>1$ be an
integer. For each $i$, set $l_i:=\rounddown{\frac{nb_i}{r_i}}$ and
define
$$\Delta_i^n:=l_ib_in-\frac{1}{2}(l_i^2+l_i)r_i,$$
which can be shown to be a non-negative integer. Define
$\Delta^n(B)=\sum_{i=1}^t \Delta_i^n$. One can verify  that
$\Delta^n(B)=\sum_i
(\frac{\overline{b_in}(r_i-\overline{b_in})}{2r_i}-
\frac{b_in(r_i-b_in)}{2r_i})$.

We set $\sigma(B):=\sum_i b_i$ and $\sigma'(B):=\sum_i
\frac{b_i^2}{r_i}$.

Given a basket $B=\{(b_i,r_i)|i=1,\cdots, t\}$ and assume that
$b_1+b_2$ is coprime to $r_1+r_2$, then we say that the new basket
$B':=\{(b_1+b_2, r_1+r_2), (b_3, r_3),\cdots, (b_t,r_t)\}$ is a
packing of $B$, denoted as $B\succ B'$. We call $B\succ B'$ a {\it
prime packing} if $b_1r_2-b_2r_1=1$. A composition of finite
packings is also called a packing. So the relation ``$\succeq$'' is
a partial ordering on the set of baskets.

\begin{setup}\label{p1}{\bf Properties  of packings.} As we have proved in
\cite{explicit}, a packing has the following properties:

Assume $B\succeq B'$. Then
\begin{itemize}
\item[i.] $\sigma(B)=\sigma(B')$ and
$\sigma'(B)\geq \sigma'(B')$;

\item[ii.] For all integer $n>1$, $\Delta^n(B)\geq
\Delta^n(B')$;
\end{itemize}
\end{setup}

\begin{setup}\label{formal}{\bf Formal baskets.} We call a pair
$(B, \tilde{P}_{-1})$ a {\it formal basket} if $B$ is
a basket and $\tilde{P}_{-1}$ is a non-negative integer. We write
$$(B, \tilde{P}_{-1})\succ (B',
\tilde{P}_{-1})$$ if $B\succ B'$.

We define some invariants of formal baskets. Considering a formal
basket ${\bf B}=(B, \tilde{P}_{-1})$, define $\tilde{P}_{-1}({\bf
B}):=\tilde{P}_{-1}$, the  volume
$$-K^3({\bf B}):=2\tilde{P}_{-1}+\sigma(B)-\sigma'(B)-6$$
and $\tilde{P}_{-2}({\bf B}):=5\tilde{P}_{-1}+\sigma(B)-10.$ So one
has
$$\tilde{P}_{-2}({\bf B})-\tilde{P}_{-1}({\bf B})=2(-K^3({\bf
B})+\sigma'(B))+2-\sigma(B).$$ For all $m\geq 2$, we define the
anti-plurigenus in an inductive way:
$$\tilde{P}_{-(m+1)}-\tilde{P}_{-m}=
\frac{1}{2}(m+1)^2(-K^3({\bf
B})+\sigma'(B))+2-\frac{m+1}{2}\sigma-\Delta^{m+1}(B).$$

Notice that $\tilde{P}_{-(m+1)}-\tilde{P}_{-m}$ is an integer
because $-K^3({\bf B})+\sigma'(B)=2\tilde{P}_{-1}+\sigma(B)-6$ has
the same parity as that of $\sigma(B)$.

Now if $B=B_X$ for a terminal weak $\bQ$-Fano 3-fold $X$ and
$\tilde{P}_{-1}=P_{-1}(X)$, then one can verify that $-K^3({\bf
B})=-K_X^3$ and $\tilde{P}_{-m}({\bf B})=P_{-m}(X)$ for all $m\geq
2$.


\end{setup}

\begin{setup}\label{p2}{\bf Properties  of packings (of formal baskets).} By \ref{p1}
and the above formulae, one can see the following immediate properties
of formal baskets:

Assume ${\bf B}:=(B, \tilde{P}_{-1})\succeq {\bf B'}:=(B',
\tilde{P}_{-1})$. Then
\begin{itemize}
\item[iii.] $-K^3({\bf B})+\sigma'(B)=-K^3({\bf B}')+\sigma'(B')$;
\item[iv.] $-K^3({\bf B})\leq -K^3({\bf B}')$;
\item[v.] $\tilde{P}_{-m}({\bf B})\leq \tilde{P}_{-m}({\bf B}')$ for all
$m\geq 2$.
\end{itemize}
\end{setup}

\begin{setup} {\bf Canonial sequence of baskets.}
Next we recall the ``canonical'' sequence of a basket $B$. Set
$S^{(0)}:=\{\frac{1}{n}|n\geq 2\}$,
$S^{(5)}:=S^{(0)}\cup\{\frac{2}{5}\}$ and inductively for all $n \ge
5$,
$$S^{(n)}:=S^{(n-1)}\cup\{\frac{b}{n}\mid \ 0<b<\frac{n}{2},\ b\
\text{coprime to}\ n\}.$$

Defined in this way, then each set $S^{(n)}$ gives a division of
the interval $(0,\frac{1}{2}]=\underset{i}\bigcup
[\omega_{i+1}^{(n)}, \omega^{(n)}_i]$ with
$\omega_{i}^{(n)},\omega_{i+1}^{(n)} \in S^{(n)}$. Let
$\omega_{i+1}^{(n)}=\frac{q_{i+1}}{p_{i+1}}$ and
$\omega^{(n)}_i=\frac{q_i}{p_i}$ with $\text{g.c.d}(q_l,p_l)=1$ for
$l=i,i+1$. Then it is easy to see that $q_ip_{i+1}-p_iq_{i+1}=1$
for all $n$ and $i$ (cf. \cite[Claim A]{explicit}).

Now given a basket ${B}=\{(b_i, r_i)| r=1,\cdots,t\}$, we would
like to define new baskets $B^{(n)}(B)$. For each $B_i=(b_i,r_i)
\in B$, if $\frac{b_i}{r_i} \in S^{(n)}$, then we set
$B^{(n)}_i:=\{(b_i,r_i)\}$. If $\frac{b_i}{r_i}\not\in S^{(n)}$,
then $\omega^{(n)}_{l+1} < \frac{b_i}{r_i} < \omega^{(n)}_{l}$ for
some $l$. We write $\omega^{(n)}_{l}=\frac{q_l}{p_l}$ and
$\omega^{(n)}_{l+1}=\frac{q_{l+1}}{p_{l+1}}$ respectively.
 In this situation, we can unpack $(b_i,r_i)$ to
$B^{(n)}_i:=\{(r_i q_l-b_ip_l) \times (q_{l+1},p_{l+1}),(-r_i
q_{l+1}+b_i p_{l+1}) \times (q_l,p_l)\}$. Adding up those
$B^{(n)}_i$, we get a new basket $B^{(n)}(B)$. $B^{(n)}(B)$ is
uniquely defined according to our construction and $B^{(n)}(B)
\succ B$ for all $n$. Notice that $B=B^{(n)}(B)$ for $n$ sufficiently large, e.g. for $n \ge \max\{r_i\}$.

In fact, we have
$$B^{(n-1)}(B)=B^{(n-1)}(B^{(n)}(B))
\succ B^{(n)}(B)$$
 for all $n\geq 1$ (cf. \cite[Claim B]{explicit}). Therefore we have a chain of baskets:
 $$ B^{(0)}(B)
\succ B^{(5)}(B) \succ ... \succ B^{(n)}(B) \succ ... \succ B. $$

The step $B^{(n-1)}(B) \succ B^{(n)}(B)$ can be achieved by a
number of  successive prime packings. Let $\epsilon_n(B)$ be the
number of such prime packings.
\end{setup}

We recall the following easy but essential properties.

\begin{lem}\label{delta}(\cite[Lemma 4.15]{explicit}) For the sequence $\{B^{(n)}(B)\}$, the following statements are true:
\begin{itemize}
\item[(i)]
$\Delta^j(B^{(0)}(B))= \Delta^j(B)$ for $j=3,4$;
\item[(ii)]
$\Delta^j(B^{(n-1)}(B))= \Delta^j(B^{(n)}(B))$ for all $j <n$;
\item[(iii)]
$\Delta^n(B^{(n-1)}(B))= \Delta^n(B^{(n)}(B))+\epsilon_n(B)$.
\end{itemize}
\end{lem}

It follows that $\Delta^j(B^{(n)}(B))=\Delta^j(B) $ for all $j \le
n$ and
$$\epsilon_n(B)=\Delta^n(B^{(n-1)}(B))-\Delta^n(B^{(n)}(B))=\Delta^n(B^{(n-1)}(B))-\Delta^n(B).$$

Moreover, given a formal basket ${\bf B}=(B, \tilde{P}_{-1})$, we
can similarly consider $B^{(n)}({\bf B}):=(B^{(n)}(B),
\tilde{P}_{-1})$. It follows that
$$\tilde{P}_{-j}(B^{(n)}({\bf B}))=\tilde{P}_{-j}({\bf
B}) \text{ for all } j \le n.$$ Therefore we can realize the
canonical sequence of formal baskets as an approximation of formal
baskets via anti-plurigenera.

\begin{setup} {\bf Solving formal baskets by anti-plurigenera.} \end{setup}
We now study the relation between formal baskets and
anti-plurigenera more closely. For a given formal basket ${\bf B}=(B,\tilde{P}_{-1})$, we begin by computing the
non-negative number $\epsilon_n$ and $B^{(0)},B^{(5)}$  in terms of $\tilde{P}_{-m}$.
{}From the definition of $\tilde{P}_{-m}$ we get:
$$\sigma(B)=10-5\tilde{P}_{-1}+\tilde{P}_{-2},$$
\begin{eqnarray*}
\Delta^{m+1}&=&(2-5(m+1)+2(m+1)^2)+\frac{1}{2}(m+1)(2-
3m)\tilde{P}_{-1}\\
&&+\frac{1}{2}m(m+1)\tilde{P}_{-2}+
\tilde{P}_{-m}-\tilde{P}_{-(m+1)}.
\end{eqnarray*}
In particular, we have:
\begin{eqnarray*}
\Delta^3&=&5-6\tp_{-1}+4\tp_{-2}-\tp_{-3};\\
\Delta^4&=&14-14\tp_{-1}+6\tp_{-2}+\tp_{-3}-\tp_{-4};\\
\end{eqnarray*}

Assume $B^{(0)}(B)=\{n_{1,r}^0\times (1,r)| r\geq 2\}$. By Lemma
\ref{delta}, we have
$$\sigma(B)=\sigma(\mB^{(0)}(B))=\sum n_{1,r}^0;$$
$$\Delta^3(\mB)=\Delta^3(\mB^{(0)}(B))=n_{1,2}^0;$$
$$\Delta^4(\mB)=\Delta^4(\mB^{(0)}(B))=2n_{1,2}^0+n_{1,3}^0.$$
Thus one gets $\mB^{(0)}$ as follows: $$\begin{cases}
n_{1,2}^0=5-6\tp_{-1}+4\tp_{-2}-\tp_{-3}\\
n_{1,3}^0=4-2\tp_{-1}-2\tp_{-2}+3\tp_{-3}-\tp_{-4}\\
n_{1,4}^0=1+3\tp_{-1}-\tp_{-2}-2\tp_{-3}+\tp_{-4}-\sigma_5\\
n_{1,r}^0=n_{1,r}^0, r\geq 5,
\end{cases}$$
where $\sigma_5:=\sum_{r\geq 5} n_{1,r}^0$. A computation gives:
$$\epsilon_5=2+\tp_{-2}-2\tp_{-4}+\tp_{-5}-\sigma_5.$$
Therefore we get $\mB^{(5)}$ as follows:
$$\begin{cases}
n_{1,2}^5=3-6\tp_{-1}+3\tp_{-2}-\tp_{-3}+2\tp_{-4}-\tp_{-5}+\sigma_5\\
n_{2,5}^5=2+\tp_{-2}-2\tp_{-4}+\tp_{-5}-\sigma_5\\
n_{1,3}^5=2-2\tp_{-1}-3\tp_{-2}+3\tp_{-3}+\tp_{-4}-\tp_{-5}+\sigma_5\\
n_{1,4}^5=1+3\tp_{-1}-\tp_{-2}-2\tp_{-3}+\tp_{-4}-\sigma_5\\
n_{1,r}^5=n_{1,r}^0, r\geq 5
\end{cases}$$
Because $\mB^{(5)}=\mB^{(6)}$, we see $\epsilon_6=0$ and on the
other hand
$$\epsilon_6=3\tp_{-1}+\tp_{-2}-\tp_{-3}-\tp_{-4}-\tp_{-5}+\tp_{-6}-\epsilon=0$$
where $\epsilon:=2\sigma_5-n_{1,5}^0 \ge 0$.

Going on a similar calculation, one gets:
\begin{eqnarray*}
\epsilon_7&=&1+\tilde{P}_{-1}+\tilde{P}_{-2}-\tilde{P}_{-5}-\tilde{P}_{-6}
+\tilde{P}_{-7}-2\sigma_5+2n^0_{1,5}+n^0_{1,6}\\
 \epsilon_{8} &=&
2\tilde{P}_{-1}+\tilde{P}_{-2}+\tilde{P}_{-3}-\tilde{P}_{-4}-\tilde{P}_{-5}
-\tilde{P}_{-7}+\tilde{P}_{-8}\\
&&-3\sigma_5+3 n^0_{1,5}+2 n^0_{1,6}+n^0_{1,7}
\end{eqnarray*}

\begin{setup}\label{geometry}{\bf Geometric inequalities.}
We say that a formal basket ${\bf B}=(B,\tilde{P}_{-1})$ is
geometric if ${\bf B}=(B_X, P_{-1}(X))$ for a terminal weak
$\bQ$-Fano 3-fold $X$.
By \cite{KMMT}, one has that
$-K_X\cdot c_2(X)\geq 0$. Therefore \cite[10.3]{YPG} gives the
following inequality:
$$\gamma(B):=\sum_{i=1}^t \frac{1}{r_i}-\sum_{i=1}^t r_i+24\geq 0
\eqno{(2.1)}$$ Furthermore $-K^3({\bf B})=-K_X^3>0$ gives the
inequality:
$$\sigma'(B)<2\tp_{-1}+\sigma(B)-6.     \eqno{(2.2)}$$

Moreover by \cite[Lemma 15.6.2]{Kollar}, whenever $P_{-m}>0$ and
$P_{-n} >0$, one has
$$ P_{-m-n} \ge P_{-m}+P_{-n}-1. \eqno{(2.3)}$$
\end{setup}


\section{\bf Plurigenus}


We begin with the following observation, which follows immediately from the definition of packing and $\gamma$:

\begin{lem}\label{31} Given a packing of baskets $B\succ
 B'$, we have $\gamma(B) > \gamma(B')$. In particular, if
inequality (2.1) doesn't hold for $B$, then it doesn't hold for $
B'$.
\end{lem}

\begin{setup}{\bf Natation and Convention.} For simplicity, we write $P_{-m}$ for $\tp_{-m}$ in what follows.

In this section, we mainly study those formal baskets $(B, P_{-1})$
satisfying inequalities (2.1) and (2.2).

We may and often do abuse the notation of ${\bf B}$ with $B$ when $\tilde{P}_{-1}$ is given.
\end{setup}


The following proposition provides an evidence about how our method
is going to work effectively.

\begin{prop}\label{finite}
Given $p_i\in \mathbb{Z}^{+}$, there are only finitely many formal
baskets admitting $(P_{-1},P_{-2})=(p_1,p_2)$ and satisfying
$(2.1)$.
\end{prop}

\begin{proof} The number of pairs in $\mB$ is upper bounded by
$\sigma=10-5p_1+p_2$. Assume $B=\{(b_i,r_i)\}$.
Then inequality (2.1) gives $$\sum_{i=1}^t r_i\leq
24+\sum_{i}\frac{1}{r_i}\leq 24+\frac{\sigma}{2}.$$ Clearly $B$ has
finite number of possibilities. This completes the proof.
%
%
\end{proof}

\begin{setup}{\bf Geometrically constrained baskets with $P_{-2}=0$.}
We now study formal baskets, satisfying (2.1) and (2.2), with
$P_{-1}=P_{-2}=0$ and will give a complete classification in this
situation. In fact, some other geometric constrains such as
$P_{-2}\geq P_{-1}$ are tacitly employed in our argument.

Given a formal basket ${\bf B}=(B,0)$ with $P_{-1}=P_{-2}=0$. The
initial basket $B^{(0)}(B)$ has datum:
$$ \begin{array}{l}
n^0_{1,2}=5-P_{-3},\\
n^0_{1,3}=4+3P_{-3}-P_{-4},\\
n^0_{1,4}=1-2P_{-3}+P_{-4}-\sigma_5.
\end{array}
$$

By Lemma \ref{31}, $B^{(0)}(B)$ satisfies (2.1) and thus
\begin{eqnarray*}
0&\leq &\gamma(B^{(0)}(B))= \sum_{r  \ge 2}(\frac{1}{r}-r)
n^0_{1,r}+24\\
&\leq& \sum_{r=2,3,4}(\frac{1}{r}-r)
n^0_{1,r}-\frac{24}{5}\sigma_5+24\\
&=&\frac{25}{12}+P_{-3}-\frac{13}{12}P_{-4}-\frac{21}{20}\sigma_5.
\end{eqnarray*}
It follows that $$P_{-4}+\sigma_5 \leq  P_{-3}+1. \eqno(3.1)$$

We need a more refined inequality, due to the fact that $B^{(5)}$
satisfies (2.1) again by Lemma \ref{31}. Because
$\gamma(B^{(5)}(B))=\gamma(B^{(0)}(B))-\frac{19}{30}\epsilon_5$, one
gets
$$ 0\leq \frac{25}{12}+P_{-3}-\frac{13}{12}P_{-4}
-\frac{21}{20}\sigma_5-\frac{19}{30}\epsilon_5.
\eqno(3.2)$$

On the other hand, by $n^0_{1,4} \ge 0$, we have
$$ P_{-4} \ge 2P_{-3}-1.$$
Thus we conclude that
$(P_{-3},P_{-4})=(0,0),(0,1),(1,1),(1,2),(2,3)$.
\end{setup}

Here is our complete classification:

\begin{thm}\label{list} Any geometric basket with $P_{-2}=0$ is
among the following list:

\centerline{Table A}

 {\tiny
$$ \begin{array}{lccccccc}
B & -K^3 & P_{-3} & P_{-4} & P_{-5} & P_{-6} & P_{-7} &P_{-8} \\
\hline\\
No.1.\ \{2 \times (1,2), 3 \times (2,5),(1,3),(1,4)\} & 1/60 & 0 & 0 & 1 & 1 & 1 &2\\
No.2.\ \{5 \times (1,2), 2 \times (1,3),(2,7),(1,4)\} & 1/84 & 0 & 1 & 0 & 1 & 1&2 \\
No.3.\ \{5 \times (1,2), 2 \times (1,3),(3,11)\} & 1/66 & 0 & 1 & 0 & 1 & 1 &2\\
No.4.\ \{5 \times (1,2), (1,3),(3,10),(1,4)\} & 1/60 & 0 & 1 & 0 & 1 & 1 &2\\
No.5.\ \{5 \times (1,2),  (1,3),2 \times (2,7)\} & 1/42 & 0 & 1 & 0 & 1 & 2&3 \\
No.6.\ \{4 \times (1,2),(2,5), 2 \times (1,3),2 \times (1,4)\} & 1/30 & 0 & 1 & 1 & 2 & 2 &4\\
No.7.\ \{3 \times (1,2),(2,5), 5 \times (1,3)\} & 1/30 & 1 & 1 & 1 & 3 & 3 &4\\
No.8.\ \{2 \times (1,2),(3,7), 5 \times (1,3)\} & 1/21 & 1 & 1 & 1 & 3 & 4 &5\\
No.9.\ \{(1,2),(4,9), 5 \times (1,3)\} & 1/18 & 1 & 1 & 1 & 3 & 4 &5\\
No.10.\ \{3 \times (1,2),(3,8), 4 \times (1,3)\} & 1/24 & 1 & 1 & 1 & 3 & 3&5 \\
No.11.\ \{3 \times (1,2),(4,11), 3 \times (1,3)\} & 1/22 & 1 & 1 & 1 & 3 & 3 &5\\
No.12.\ \{3 \times (1,2),(5,14), 2 \times (1,3)\} & 1/21 & 1 & 1 & 1 & 3 & 3 &5\\
No.13.\ \{2 \times (1,2),2 \times (2,5),4 \times (1,3)\} & 1/15 & 1 & 1 & 2 & 4 & 5 &7\\
No.14.\ \{ (1,2),(3,7),  (2,5),4 \times (1,3)\} & 17/210 & 1 & 1 & 2 & 4 & 6 &8\\
No.15.\ \{2 \times (1,2), (2,5),(3,8),3 \times (1,3)\} & 3/40 & 1 & 1 & 2 & 4 & 5&8 \\
No.16.\ \{2 \times (1,2),(5,13),3 \times (1,3)\} & 1/13 & 1 & 1 & 2 & 4 & 5 &8\\
No.17.\ \{ (1,2),3 \times (2,5),3 \times (1,3)\} & 1/10 & 1 & 1 & 3 & 5 & 7 &10\\
No.18.\ \{4 \times (1,2), 5 \times (1,3),(1,4)\} & 1/12 & 1 & 2 & 2 & 5 & 6 &9\\
No.19.\ \{4 \times (1,2), 4 \times (1,3),(2,7)\} & 2/21 & 1 & 2 & 2 & 5 & 7 &10\\
No.20.\ \{4 \times (1,2), 3 \times (1,3),(3,10)\} & 1/10 & 1 & 2 & 2 & 5 & 7 &10\\
No.21.\ \{3 \times (1,2),(2,5), 4\times (1,3),(1,4)\} & 7/60 & 1 & 2 & 3 & 6 & 8 &12\\
No.22.\ \{3 \times (1,2), 7 \times (1,3)\} & 1/6 & 2 & 3 & 4 & 9 & 12 &17 \\
No.23.\ \{2 \times (1,2),(2,5), 6 \times (1,3)\} & 1/5 & 2 & 3 & 5 & 10 & 14 &20\\
\end{array} $$}
\end{thm}
\begin{proof}
This theorem follows from Propositions \ref{00}, \ref{01}.
\end{proof}

\begin{prop}\label{00} If
$(P_{-3},P_{-4})=(0,0)$, then $B$ is of type No.1 in Table A.
\end{prop}
\begin{proof}

Now $\sigma_5 \le 1$ and $\epsilon_5=2+P_{-5}-\sigma_5 \le 3$ by
$(3.2)$.

\begin{CLM}\label{c00} The situation $(\sigma_5, P_{-5})=(0,0)$ does not happen.
\end{CLM}
\begin{proof}[Proof of the claim]
We have $$B^{(5)}(B)=\{3 \times (1,2), 2 \times (2,5), 2 \times
(1,3), (1,4)\}.$$

If $B=B^{(5)}(B)$, then
$-K^3(B)=\sigma-\sigma'-6=-\frac{1}{60}<0$, a contradiction.
Thus $B\neq B^{(5)}(B)$. Assume that $B$ has totally $t$ pairs. Then
$t<8$ since $B^{(5)}\succ B$. From $B^{(5)}(B)$ we know $\sum_i
r_i=26$. Thus (2.1) becomes $\sum_{i=1}^t\frac{1}{r_i}\geq 2$.
Assume $r_1\leq r_2\leq \cdots \leq r_t$. If $t\leq 4$, then $r_t>2$
and $\sum \frac{1}{r_i}\leq \frac{3}{2}+\frac{1}{r_t}<2$. So (2.1)
fails. If $t=5$, we consider the value of $r_3$. Whenever $r_3=2$,
one has $r_4\geq 3$ and $26=6+r_4+r_5\leq 6+2r_5$ gives $r_5\geq
10$. Thus $\sum_i
\frac{1}{r_i}\leq\frac{3}{2}+\frac{1}{3}+\frac{1}{10}<2$. So (2.1)
fails. Whenever $r_3\geq 3$, then $r_4\geq 3$ and $r_5\geq 7$. So
again $\sum_i \frac{1}{r_i}\leq 1+\frac{2}{3}+\frac{1}{7}<2$, a
contradiction to (2.1). Therefore we have seen $t=6,7$, which means
that $B$ is exactly obtained by 1 or 2 prime packings from
$B^{(5)}(B)$.

When $t=7$, $B$ must be one of the following cases:
\begin{itemize}
\item[(I).] $\{2\times (1,2), (3,7), (2,5), 2\times (1,3),
(1,4)\}$; $-K^3=-\frac{1}{60}<0$ (contradiction);
\item[(II).] $\{3\times (1,2), (2,5), (3,8), (1,3), (1,4)\}$;
$-K^3=-\frac{1}{120}<0$ (contradiction);
\item[(III).] $\{3\times (1,2), 2\times (2,5), (1,3), (2,7)\}$;
$-K^3=-\frac{1}{210}<0$ (contradiction);
\end{itemize}

When $t=6$, $B$ is nothing but an extra prime packing from one of
I,II and III:

\begin{itemize}
\item[(I-1).] $\{(1,2), 2\times (3,7),  2\times (1,3),
(1,4)\}$; $\sum_i \frac{1}{r_i}<2$ (contradiction);

\item[(I-2).] $\{(1,2), (4,9),  (2,5), 2\times (1,3),
(1,4)\}$; $\sum_i \frac{1}{r_i}<2$ (contradiction);

\item[(I-3).] $\{2\times (1,2), (5,12),  2\times (1,3),
(1,4)\}$; $-K^3=0$ (contradiction);

\item[(I-4).] $\{2\times (1,2), (3,7), (3,8), (1,3),
(1,4)\}$; $\sum_i \frac{1}{r_i}<2$ (contradiction);

\item[(I-5).] $\{2\times (1,2), (3,7),  (2,5), (1,3),
(2,7)\}$; $\sum_i \frac{1}{r_i}<2$ (contradiction);

\item[(II-1).] $\{3\times (1,2), (5,13), (1,3), (1,4)\}$;
$-K^3=-\frac{1}{156}<0$ (contradiction);

\item[(II-2).] $\{3\times (1,2), (2,5), (4,11), (1,4)\}$;
$-K^3=-\frac{1}{220}<0$ (contradiction);

\item[(II-3).] $\{3\times (1,2), (2,5), (3,8), (2,7)\}$;
$\sum_i \frac{1}{r_i}<2$ (contradiction);

\item[(III-1).] $\{3\times (1,2), 2\times (2,5), (3,10)\}$;
$-K^3=0$ (contradiction);
\end{itemize}
\end{proof}

We go on proving Proposition \ref{00}.
\medskip

If $\sigma_5=0$ and $P_{-5}>0$. Because $2+P_{-5}=\epsilon_5\le 3$,
we see $P_{-5}=1$ and $B^{(5)}(B)=\{2 \times (1,2), 3 \times (2,5),
(1,3),(1,4)\}$. A computation shows that any non-trivial packing of
$B^{(5)}$ has $\gamma <0$. Hence $B=B^{(5)}(B)$. So $B$ corresponds to
case No.1 in Table A.
\medskip

 If $\sigma_5=1$ and $P_{-5}=0$, then
we have $B^{(5)}(B)=\{4  \times (1,2), (2,5), 3 \times (1,3), (1,s)
\}$ with $s\geq 5$, $-K^3=\frac{1}{5}-\frac{1}{s}$ and
$\gamma=5-s+\frac{1}{5}+\frac{1}{s}$. When $s \ge 6$, we have
$\gamma<0$, a contradiction. Hence we must have $s=5$. Since
$-K^3(B^{(5)})=0$, so $B^{(5)}(B) \succ B$ is nontrivial. However,
any non-trivial packing of $B^{(5)}(B)$ has $\gamma <0$, which
still gives a contradiction. Thus  this case can not happen.
\medskip

Finally if $\sigma_5=1$ and $P_{-5}>0$, then we get a contradiction
from $(3.2)$.

We have proved Proposition \ref{00}.
\end{proof}

\begin{prop}\label{01} (1)
If $(P_{-3},P_{-4})=(0,1)$, then $B$ is of type No.2-No.6 in Table
A;

(2) If $(P_{-3},P_{-4})=(1,1)$, then $B$ is of type No.7-No.17 in
Table A;

(3) If  $(P_{-3},P_{-4})=(1,2)$, then $B$ is of type No.18-No.21 in
Table A;

(4) If $(P_{-3},P_{-4})=(2,3)$, then $B$ is of type No. 22, No. 23
in Table A.
\end{prop}
\begin{proof} (1)
By $(3.1)$, we have $\sigma_5=0$, hence $B^{(0)}(B)=\{5 \times
(1,2), 3 \times (1,3), 2 \times (1,4)\}$. By $(3.2)$, we have
$P_{-5}=\epsilon_5 \le 1$.

If $P_{-5}=0$, then we can easily compute all possible formal
baskets $B$ with $B^{(5)}(B)=\{5 \times (1,2), 3 \times (1,3), 2
\times (1,4)\}$, $\gamma>0$ and $-K^3(B)>0$. In fact, by classifying all baskets with $B^{(5)}(B)$ as above, and
verifying inequalities $(2.1),(2.2)$,
the reader should have no difficulty to see that  $B$ has 4
types which correspond to No.2 through No.5 in Table A.

If $P_{-5}=1$, then $B^{(5)}(B)=\{4 \times (1,2),(2,5), 2 \times (1,3),
2 \times (1,4)\}$. Because any basket dominated by $B^{(5)}(B)$ has
$\gamma <0$, we see $B=B^{(5)}(B)$ which corresponds to No. 6 in Table
A.
\medskip

(2) In this case, $0\leq n^0_{1,4}=-\sigma_5 $ gives $\sigma_5=0$.
By $(3.2)$, we have $\epsilon_5 \le 3$.

If $\epsilon_5=0$, then we get $B=B^{(5)}(B)=\{4 \times (1,2), 6 \times
(1,3)\}$ with $-K^3(B)=0$, a contradiction.

If $\epsilon_5=1$, then we get $B^{(5)}(B)=\{3\times (1,2), (2,5),
5\times (1,3)\}$. By computation, we see that $B$ corresponds to
 No. 7 through No. 12 in Table A.

If $\epsilon_5=2$, then we  get $B^{(5)}(B)=\{2\times (1,2), 2\times
(2,5), 4\times (1,3)\}$. We see that  $B$ corresponds to No. 13
through No. 16 in Table A.

If $\epsilon_5=3$, then we  get $B^{(5)}(B)=\{(1,2), 3\times (2,5),
3\times (1,3)\}$. We see that $B$ has only one possibility, which is $B^{(5)}(B)$ corresponding to No. 17
in Table A.
\medskip

(3) By $(3.1)$, we must have $\sigma_5=0$. Moreover, by $(3.2)$, we
have $\epsilon_5 \le 1$.

If $\epsilon_5=0$, then we get $B^{(5)}(B)=\{4\times (1,2), 5\times
(1,3), (1,4)\}$ and $B$ corresponds to
 No. 18 through No. 20 in Table A.

If $\epsilon_5=1$, then we  get $B^{(5)}(B)=\{3\times (1,2), (2,5),
4\times (1,3), (1,4)\}$ and $B=B^{(5)}(B)$ corresponds to No. 21 in Table A.
\medskip

(4) By $(3.1)$, we must have $\sigma_5=0$. Moreover, by $(3.2)$, we
have $\epsilon_5 \le 1$.

If $\epsilon_5=0$, then we  get $B^{(5)}(B)=\{3\times (1,2), 7\times
(1,3)\}$ and $B=B^{(5)}(B)$ corresponds to No.22 in Table A.

If $\epsilon_5=1$, then we  get $B^{(5)}(B)=\{2\times (1,2), (2,5),
6\times (1,3)\}$ and $B=B^{(5)}(B)$ corresponds to No.23 in Table A.
\end{proof}

Now we are able to prove the following

\begin{thm}\label{6} Let $X$ be a terminal weak $\bQ$-Fano 3-fold. Then
$P_{-6} >0$.
\end{thm}

\begin{proof} Set $B:=B_X$.
If $P_{-6}=0$, then $P_{-1}=P_{-2}=P_{-3}=0$. By $\epsilon_6=0$, we
get $P_{-4}=P_{-5}=\epsilon=0$. Thus $B^{(5)}(B)=\{3 \times (1,2), 2
\times (2,5), 2 \times (1,3), (1,4)\}$. By Claim 1 in the proof of Proposition \ref{00}, we know that
such a basket $B$ does not exist. Thus $P_{-6}(X)>0$.
\end{proof}

\begin{prop}\label{>0} Let $X$ be a terminal weak $\bQ$-Fano 3-fold.
Then $P_{-4}>0$ unless $B_X=\{2 \times (1,2), 3 \times (2,5),
(1,3),(1,4)\}$.
\end{prop}

\begin{proof}
Assume $P_{-4}=0$. Then clearly $P_{-1}=P_{-2}=0$. Recall that
$n^0_{1,4}=1-2P_{-3}+P_{-4}-\sigma_5$. It follows that $P_{-3}=0$.
By Proposition \ref{00},  we see $B=\{2 \times (1,2), 3 \times
(2,5), (1,3),(1,4)\}$.
\end{proof}

\begin{prop} \label{p-8} Let $X$ be a terminal weak $\bQ$-Fano 3-fold.
If $P_{-2}>0$, then $P_{-2k} \geq 2$ for all $k \ge 4$.
\end{prop}

\begin{proof}
%
%
%
If $P_{-2}\geq 2$, then there is nothing to prove. Thus it remains
to consider the case $P_{-2}=1$. (Actually we will prove that
$P_{-6} \geq 2$ except for a very special case.)
\medskip

\noindent {\bf Case 1.} $P_{-1}=0$.

Then $n^0_{1,4}=-2P_{-3}+P_{-4}-\sigma_5 \ge 0$. Note that
$P_{-4}\geq 2$ whenever $P_{-3}
>0$. We only need to consider the case $P_{-3}=0$ and $P_{-4}=1$.
Since $n_{1,4}^0=1-\sigma_5$, we see $\sigma_5\leq 1$.

If $\sigma_5=0$, then $\epsilon=0$ and $\epsilon_5=1+P_{-5} \leq
n^0_{1,3}=1$. Thus $P_{-5}=0$. Now $\epsilon_6=0$ gives $$2\leq
P_{-2}+P_{-6}=P_{-3}+P_{-4}+P_{-5}+\epsilon =1,$$ a contradiction.
Thus $\sigma_5=1$.

Now if $P_{-5}>0$, then $ P_{-3}+P_{-4}+P_{-5}+\epsilon \geq 3$ and
thus $\epsilon_6=0$ gives $P_{-6} \ge 2$. Clearly $P_{-8}\geq 2$.

If $P_{-5}=0$, then $\epsilon_5=0$ and $B^{(5)}(B)=\{9 \times (1,2),
(1,3), (1,s)\}$ with $s \geq 5$. By our definition,
$B^{(n)}(B)=B^{(5)}(B)$ for all $n \ge 5$. Also notice that $B=B^{(n)}(B)$ for $n$ sufficiently large.
We thus have $B=B^{(5)}(B)$.

When $\epsilon=1$, then $n_{1,5}^0=1$ and $n_{1,r}^0=0$ for all
$r\geq 6$, which means $s=5$. Now
$\sigma'(B)=\frac{9}{2}+\frac{1}{3}+\frac{1}{5}>5$ and (2.2) fails. Thus we have
$\epsilon\geq 2$. Hence $\epsilon_6=0$ implies $P_{-6}=\epsilon\geq 2$.
\medskip

\noindent {\bf Case 2.} $P_{-1}=1$.

We may assume $P_{-6}=1$. Then $P_{-2}=P_{-3}=P_{-4}=P_{-5}=1$.
Since $\epsilon_6=0$, one gets $\epsilon=2$ and therefore $\sigma_5 >0$. Note that
$\epsilon_5=2-\sigma_5 \ge 0$. We have $\sigma_5 \leq 2$.

If $\sigma_5=2$, then $n^0_{1,5}=2$. We have $B^{(5)}(B)=\{2 \times
(1,2), 2 \times (1,3), 2 \times (1,5)\}$. By the same reason as above,
$B=B^{(5)}(B)$.
Because $\sigma'(B)=1+\frac{2}{3}+\frac{2}{5}>2$, so (2.2) fails.

Thus $\sigma_5=1$, and then we have $B^{(5)}(B)=\{ (1,2), (2,5), (1,3),
(1,4),(1,s)\}$ with $s \ge 6$.
If $s \ge 8$, then $\epsilon_8 \ge 0$ gives
$$P_{-8}-P_{-7}=\epsilon_8+1  \ge 1.$$
Since $P_{-7} \ge P_{-6} \ge 1$, we have $P_{-8} \ge 2$.

We now assume that $s=6$, $7$. Considering all baskets with given
$B^{(5)}$, we may find that they dominate one of the following
minimal elements:

$B_1=\{(3,7), (2,7), (1,s)\}$;

$B_2=\{(1,2), (3,8), (1,4), (1,s)\}$.

Because $\sigma'(B)\geq \sigma'(B_i)\geq 2$ whenever $s=6,7$ and
$i=1,2$, inequality (2.2) fails for all $B$, which says that this
case does not happen.

We have actually proved $P_{-8}\geq 2$. Furthermore $P_{-6}\geq 2$
except when $P_{-1}=1$ and $\sigma_5=1$.

This completes the proof.
\end{proof}

Now we prove the following:
\begin{thm}\label{>2} Let $X$ be a terminal weak $\bQ$-Fano 3-fold.
Then $P_{-2k} \ge 2$ for all $k \geq 4$.
\end{thm}

\begin{proof}
When $P_{-2} >0$, then this follows from Proposition \ref{p-8}.

When $P_{-2}=0$, then it follows from Theorem \ref{list} and by
computing $P_{-2k}$ for each case in the list.
\end{proof}

\section{\bf The anti-volume}

By Riemann-Roch formula directly, we have
\begin{eqnarray*}
\frac{1}{2}(-K^3) &=& P_{-1}-3+l(2),\\
\frac{5}{2}(-K^3)&=& P_{-2}-5+l(3).
\end{eqnarray*}



\bigskip

%


\begin{setup}\label{-K^3}{\bf An inequality.} We have $B^{(0)}(B) \succ
B$ and so $(B^{(0)}(B),P_{-1})\succ (B,P_{-1})$. By our formulae in
Section 2, we get
$$\sigma'(B^{(0)})-K^3(B^{(0)})=\sigma'(B)-K^3(B)=2P_{-1}+\sigma(B)-6.$$

We have
\begin{eqnarray*}
\sigma'(B^{(0)}(B))&=& \frac{1}{2}n_{1,2}^0
+\frac{1}{3}n_{1,3}^0+\frac{1}{4}n_{1,4}^0+\sum_{\sigma_5}\frac{1}{r}n_{1,r}^0\\
&\leq & \frac{1}{2}n_{1,2}^0
+\frac{1}{3}n_{1,3}^0+\frac{1}{4}(n_{1,4}^0+\sigma_5)\\
&=& \frac{1}{12}(49-35P_{-1}+13P_{-2}-P_{-4}).
\end{eqnarray*}

Thus we get the following inequality by \ref{p2}(iv):
\begin{eqnarray*}
-K^3(B)&\geq& -K^3(B^{(0)}(B))= 2P_{-1}+\sigma(B)-6-\sigma'(B^{(0)}(B))\\
&\geq& \frac{1}{12} (-1-P_{-1}-P_{-2}+P_{-4}).\hskip3.5cm (4.1)
\end{eqnarray*}

In particular, we have $-K^3 \geq \frac{1}{12} $ whenever $P_{-4} >
P_{-2}+P_{-1}+1$.
\end{setup}

\begin{lem}\label{=} Assume
$P_{-4} = P_{-2}+P_{-1}+1$. Then:
\begin{itemize}
\item[(1).] $-K^3 \geq \frac{1}{20}$ when $\sigma_5
>0$;
\item[(2).] $-K^3 \geq \frac{1}{30}$ when $\epsilon_5 >0$.
\end{itemize}
\end{lem}
\begin{proof} If $\sigma_5>0$, then our computation in \ref{-K^3}
gives:
$$-K^3(B)\geq -K^3(B^{(0)}(B))\geq
\frac{1}{4}\sigma_5-\sum_{\sigma_5}^{r\geq 5} \frac{1}{r}\geq
\frac{1}{20}.$$

If $\epsilon_5>0$ and $\sigma_5=0$, then
\begin{eqnarray*}
\sigma'(B^{(5)}(B))&=& \sigma'(B^{(0)}(B))-\frac{1}{30}\epsilon_5\\
&\leq& \sigma'(B^{(0)}(B))-\frac{1}{30}.
\end{eqnarray*}

Therefore $-K^3(B)\geq -K^3(B^{(5)}(B))\geq
-K^3(B^{(0)}(B))+\frac{1}{30}\geq \frac{1}{30}$.
\end{proof}

\begin{setup}\label{assumption}{\bf Assumption.} Under the situation
$P_{-4}=P_{-2}+P_{-1}+1$, we only need to study the case
$\sigma_5=\epsilon_5=0$.
\end{setup}

Now we are prepared to prove the following:

\begin{thm} Let $X$ be a terminal weak $\bQ$-Fano 3-fold. Then
$$-K^3(X) \geq \frac{1}{330}.$$
\end{thm}
\begin{proof} By (4.1) and Lemma \ref{=}, we only need to study one
of the following situations:

(i) $P_{-4}<P_{-2}+P_{-1}+1$;

(ii) $P_{-4}=P_{-2}+P_{-1}+1$, $\sigma_5=0$ and $\epsilon_5=0$.
\medskip

\noindent {\bf Case I.} $P_{-1}=0$.  We have $\sigma=10+P_{-2}\geq
10$.
\medskip

{\bf Subcase I-1.} $P_{-2} \geq 3$. Then $P_{-4} \geq
2P_{-2}-1>P_{-2}+P_{-1}+1$. By (4.1), we have $-K^3 \geq
\frac{1}{12}$.



\medskip

{\bf Subcase I-2.} $P_{-2}=2$. Then $n^0_{1,3}\geq 0$ and $n^0_{1,4}
\ge 0$ gives
$$ 3P_{-3} \ge P_{-4} \ge 1+2P_{-3},$$
which implies that  $P_{-3} \ge 1$.

If $P_{-3} \ge 2$, then $P_{-4} \ge 5>P_{-2}+P_{-1}+1$. By $(4.1)$,
we see  $-K^3(B)\geq \frac{1}{12}$.

If $P_{-3}=1$, then $P_{-4}=3$. We have $B^{(0)}(B)=\{12 \times
(1,2)\}$. Clearly $B^{(0)}(B)$ admits no packing. So $B=B^{(0)}(B)$.
However $-K^3(B)=0$, a contradiction.
\medskip

{\bf Subcase I-3.} $P_{-2}=1$.
By $n^0_{1,4}\geq 0$ and $n^0_{1,3} \ge 0$, we get
$$2+3P_{-3} \ge P_{-4} \ge 2P_{-3}.$$
Also, if $P_{-4} \ge 3$, then (4.1) gives $-K^3(B) \geq \frac{1}{12}$.
 Thus we only need to consider the situations:
$(P_{-3},P_{-4})=(0,1),(0,2),(1,2)$.

If $(P_{-3},P_{-4})=(1,2)$, then $B^{(0)}(B)=\{8 \times (1,2), 3 \times
(1,3)\}$ with $-K^3(B^{(0)})=0$. Thus $B$ must be a packing of
$B^{(0)}$. The one-step packing $B_1=\{7 \times (1,2),(2,5),2
\times (1,3)\}$ has $-K^3(B_1)=\frac{1}{30}>0$. Because $B_1\succ B$, we see
$-K^3(B)\geq \frac{1}{30}$.

If $(P_{-3},P_{-4})=(0,2)$, then $P_{-4}=P_{-2}+P_{-1}+1$. By our
assumption, we may assume $\sigma_5=\epsilon_5=0$. So
$B^{(0)}(B)=\{9 \times (1,2), 2\times (1,4)\}$. Since $B^{(0)}$ admits
no prime packing, $B=B^{(0)}(B)$ and $-K^3(B)=0$, a contradiction.

Finally we consider  the situation $(P_{-3},P_{-4})=(0,1)$.
$n_{1,4}^0\geq 0$ gives $\sigma_5\leq 1$. If $\sigma_5=0$, we have
$B^{(0)}(B)=\{9 \times (1,2), (1,3),(1,4)\}$ with
$-K^3(B^{(0)})=-\frac{1}{12}<0$. Considering a minimal basket
$B_{min}$ dominated by $B^{(0)}(B)$, then either
$B_{min}=\{(10,21),(1,4)\}$ with $-K^3(B_{min})=-\frac{1}{84}<0$ or
$B_{min}=\{ 9 \times (1,2), (2,7)\}$ with $-K^3=-\frac{1}{14}<0$.
Thus we see $-K^3(B)\leq -K^3(B_{min})<0$, a contradiction.

If $\sigma_5=1$, we see $B^{(0)}(B)=\{9 \times (1,2),
(1,3),(1,s)\}$ with $s\geq 5$ and
$-K^3(B^{(0)}(B))=\frac{1}{6}-\frac{1}{s}$. Notice that baskets
dominated by $B^{(0)}(B)$ are linearly ordered.

For $s \ge 7$, we have $-K^3(B)\geq -K^3(B^{(0)}(B)) \ge
\frac{1}{42}$. For $s=6$, $-K^3(B^{(0)})=0$ and the one step
packing is $B_1= \{8\times (1,2), (2,5),(1,6)\}$ with
$-K^3=\frac{1}{30}$. Thus $-K^3(B)\geq -K^3(B_1)\geq \frac{1}{30}$.
For the last case $s=5$, $B$ must be dominated by $B_2=\{7\times
(1,2), (3,7),(1,5)\}$ with $-K^3(B_2)=\frac{1}{70}$. We see
$-K^3(B)\geq -K^3(B_2)\geq \frac{1}{70}$.
\medskip

{\bf Subcase I-4.} $P_{-2}=0$. By Proposition \ref{list}, we know
$-K^3(B)\geq \frac{1}{84}$.


This completes the proof for Case I.
\medskip

\noindent {\bf Case II.} $P_{-1}=1$. We have $\sigma=5+P_{-2}\geq
5$.
\medskip

{\bf Subcase II-1.} $P_{-2} \ge 4.$ Then $P_{-4} \ge 2P_{-2}-1 \ge
P_{-2}+3>P_{-2}+P_{-1}+1$ by (2.3). According to (4.1), we have
$-K^3(B)\geq \frac{1}{12}$.
\medskip

{\bf Subcase II-2.} $P_{-2} =3.$ Then $P_{-4} \ge
5=P_{-2}+P_{-1}+1$. By our assumption, we only need to consider the
situation $P_{-4}=5$ and $\sigma_5=\epsilon_5=0$. Now $n_{1,4}^0\geq
0$ gives $P_{-3}=3$ and thus $B^{(0)}(B)=\{8 \times (1,2)\}$ with
$-K^3(B^{(0)}(B))=0$. Since $B^{(0)}(B)$ is already minimal, $B=B^{(0)}(B)$
and thus $-K^3(B)=0$, a contradiction.
\medskip

{\bf Subcase II-3.} $P_{-2} =2.$ Notice that $P_{-3} \ge P_{-2}=2$
and $P_{-4} \geq 2P_{-2}-1=3$. In fact, if $P_{-4}\geq
5>P_{-2}+P_{-1}+1$, we have $-K^3(B)\geq \frac{1}{12}$. {}From
$n^0_{1,3}\geq 0$ and $n^0_{1,4} \ge 0$, we get $3P_{-3}-2 \ge
P_{-4} \ge 2P_{-3}-2$. So it suffices to consider the following
situations: $(P_{-3},P_{-4})=(2,3),(2,4),(3,4)$.

If $(P_{-3},P_{-4})=(3,4)$, then $B^{(0)}(B)=\{4 \times (1,2), 3 \times
(1,3) \}$ with $-K^3=0$. Consider the one step packing $B_1$ of
$B^{(0)}$, one sees $B_1= \{3 \times (1,2), (2,5),2 \times (1,3)\}$
with $-K^3(B_1)=\frac{1}{30}>0$. Thus $-K^3(B)\geq -K^3(B_1)\geq
\frac{1}{30}$.

If $(P_{-3},P_{-4})=(2,4)$, then we may assume that $\sigma_5=0$
since $P_{-4}=P_{-2}+P_{-1}+1$. Thus $B^{(0)}(B)=\{5 \times (1,2), 2
\times (1,4) \}$ with $-K^3(B^{(0)}(B))=0$. Because $B^{(0)}(B)$ is
minimal, $B=B^{(0)}(B)$, a contradiction.

If $(P_{-3},P_{-4})=(2,3)$, then $n_{1,4}^0\geq 0$ gives
$\sigma_5\leq 1$. Thus either $B^{(0)}(B)=\{5 \times (1,2), (1,3),
(1,4)\}$ with $-K^3(B^{(0)})<0$ or $B^{(0)}=\{5 \times (1,2), (1,3),
(1,s)\}$ with $s \ge 5$.

We consider the first case. One can check that any minimal basket
dominated by $B^{(0)}$ has negative anti-volume. Thus this case can
not happen at all.

Now we consider the later case. If $s \ge 7$, then $$-K^3(B)\geq
-K^3(B^{(0)}(B)) \ge \frac{1}{42}.$$ If $s=6$, then the one-step
packing $B_1=\{4\times (1,2), (2,5), (1,6)\}$ with $-K^3(B)\geq
-K^3(B_1)=\frac{1}{30}$. If $s=5$, then the two-step packing $B_2=
\{3 \times (1,2), (3,7), (1,5)\}$ with $-K^3(B)\geq
-K^3(B_2)=\frac{1}{70}$.
\medskip

{\bf Subcase II-4.} $P_{-2} =1.$  We get $\sigma=6$ and
$-K^3+\sigma'=2$. For a similar reason, we only need to consider
the situation $P_{-4} \le 3$. So it remains to consider the cases:
$$(P_{-3},P_{-4})=(1,1),(1,2),(1,3),(2,2),(2,3),(3,3)$$ since
$P_{-4} \ge P_{-3}$ by $n_{1,4}^0\geq 0$.
\medskip

{\bf II-4a.} If $(P_{-3},P_{-4})=(3,3)$, then $B^{(0)}(B)=\{6 \times
(1,3)\}$ with $-K^3(B^{(0)})=0$. No further packing is allowed, so
$B=B^{(0)}(B)$, a contradiction.
\medskip

{\bf II-4b.} If $(P_{-3},P_{-4})=(2,3)$, by our assumption, we may
assume $\sigma_5=0$ and $B^{(0)}(B)=\{ (1,2),3 \times (1,3), 2 \times
(1,4) \}$ with $-K^3=0$. The one-step packing $B_1$ is either
$\{(2,5), 2\times (1,3), 2\times (1,4)\}$ with
$-K^3(B_1)=\frac{1}{30}$ or $\{(1,2), 2\times (1,3), (2,7), (1,4)\}$
with $-K^3(B_1)=\frac{1}{84}$. Thus we see $-K^3(B)\geq
\frac{1}{84}$.
\medskip

{\bf II-4c.} If $(P_{-3},P_{-4})=(1,3)$, by our assumption, we may
assume $\sigma_5=0$ and thus $B^{(0)}(B)=\{2 \times (1,2), 4 \times
(1,4) \}$ with $-K^3=0$. This allows no further packings and so $
B=B^{(0)}(B)$, a contradiction.
\medskip

{\bf II-4d.} If $(P_{-3},P_{-4})=(2,2)$, then $\sigma_5\leq 1$ by
$n_{1,4}^0\geq 0$. So either $B^{(0)}(B)=\{(1,2),4 \times (1,3),
(1,4)\}$ or $B^{(0)}(B)=\{(1,2), 4 \times (1,3),  (1,s)\}$ with $s \ge
5$.

In the first situation, every packing of $B^{(0)}(B)$ has negative
$-K^3$, which is absurd. Actually, it suffices to check that one
minimal basket $\{(5,14),(1,4)\}$ has $-K^3=\frac{-1}{28}$ and that
the other minimal basket $\{(1,2),(5,16)\}$ has
$-K^3=-\frac{1}{16}$.

In the last situation, we consider the value of $s$. If $s \ge 7$,
then $-K^3(B)\geq -K^3(B^{(0)}(B)) \ge \frac{1}{42}$. If $s=6$, then
the one-step packing is $\{(2,5), 3\times (1,3), (1,6)\}$ with
$-K^3(B)\geq \frac{1}{30}$. If $s=5$, the one-step packing has
$-K^3=0$, but the two-step packing is $\{(3,8),2 \times (1,3),
(1,5)\}$ with $-K^3=\frac{1}{120}$. Hence any further packing gives $-K^3(B)\geq \frac{1}{120}$.
\medskip

{\bf II-4e.} If $(P_{-3},P_{-4})=(1,2)$, then $n_{1,4}^0 \geq 0$ gives
$\sigma_5\leq 3$.

If $\sigma_5 \ge 2$, then $\sigma'(B^{(0)})\leq
1+\frac{1}{3}+\frac{1}{4}+\frac{2}{5}<2$ and thus $-K^3(B)\geq
-K^3(B^{(0)}(B))\geq \frac{1}{60}$.

If $\sigma_5=0$, then $B^{(0)}(B)=\{2\times (1,2), (1,3), 3\times
(1,4)\}$ with \newline $-K^3(B^{(0)}(B))<0$. Because the only two
minimal elements of $B^{(0)}(B)$ are $\{(3,7), 3 \times (1,4)\}$
and $\{2 \times (1,2), (4,15)\}$ with both negative $-K^3$, so this
case does not happen at all.

If $\sigma_5 =1 $, we have $B^{(0)}(B)=\{2 \times (1,2), (1,3), 2
\times (1,4), (1,s)\}$ with $s\geq 5$. When $s \ge 7$, then
$-K^3(B)\geq -K^3(B^{(0)}) \ge \frac{1}{42}$. When $s=6$, then
$-K^3(B^{(0)}(B))=0$, but the one-step packing of $B^{(0)}(B)$ is either
$\{(1,2),(2,5),2 \times (1,4), (1,6)\}$ with $-K^3(B)\geq
\frac{1}{30}$ or $$\{2 \times (1,2),(2,7), (1,4),  (1,6)\}$$ with
$-K^3(B)\geq \frac{1}{84}$. When $s=5$, then we have  $$B^{(0)}(B)=\{2
\times (1,2), (1,3), 2 \times (1,4),  (1,5)\}$$ with negative
$-K^3$. By computing all possible pakings, one can find out that $B$
can be obtained by packing either $\{(3,7), 2 \times (1,4),(1,5)\}$
with $-K^3=\frac{1}{70}$ or $\{(1,2),(2,5), (1,4), (2,9)\}$ with
$-K^3=\frac{1}{180}$. Thus we have proved
$-K^3(B)\geq\frac{1}{180}$.
\medskip

{\bf II-4f.} If $(P_{-3},P_{-4})=(1,1)$, then $n_{1,4}^0\geq 0$
gives $\sigma_5\leq 2$.

If $\sigma_5=0$, then $B^{(0)}=\{2\times (1,2), 2\times (1,3),
2\times (1,4)\}$. By calculation, one sees that all minimal elements
dominated by $B^{(0)}$ have negative $-K^3$. Thus this case doesn't
happen.

If $\sigma_5=1$, then $B^{(0)}=\{2\times (1,2), 2\times (1,3),
(1,4), (1,s)\}$ with $s\geq 5$. When $s=5$, each minimal element
dominated by $B^{(0)}$ has negative $-K^3$. In fact, they are $\{2
\times (2,5),(2,9)\}$ and $\{2 \times (1,2),  (3,10),(1,5)\}$. Thus
this case doesn't happen.

When $s\geq 6$, by calculation, we see that $B$ is dominated by one
of the following baskets and thus $-K^3(B)$ has the lower bounds
accordingly:
\begin{itemize}
\item $\{2 \times (1,2), 2 \times (1,3),(1,4), (1,s)\}$ with $s\geq 13$ and
$-K^3=\frac{1}{12}-\frac{1}{s}\geq \frac{1}{156}$. But when $s \ge 13$, $\gamma<0$. So this case can not happen;

\item $\{2 \times (1,2), (1,3), (2,7), (1,s)\}$ with $s=11,12$ and $-K^3=
\frac{2}{21}-\frac{1}{s}\geq \frac{1}{231}$;


\item $\{(1,2),(2,5), (1,3),(1,4), (1,s)\}$ with $s= 9,10,11,12$ and
$-K^3=\frac{7}{60}-\frac{1}{s}\geq \frac{1}{180}$;

\item $\{(3,7), (1,3), (1,4), (1,8)\}$ with $-K^3=\frac{1}{168}$;

\item $\{(1,2),(2,5), (2,7), (1,8)\}$ with  $-K^3=
\frac{1}{280}$;




\item $\{2\times (2,5), (1,4), (1,s)\}$ with $s=7,8,9,10,11,12$ and
$-K^3=\frac{3}{20}-\frac{1}{s}\geq  \frac{1}{140}$.
\end{itemize}
\medskip

Finally if $\sigma_5=2$, $B^{(0)}(B)=\{2\times (1,2), 2\times (1,3),
(1,s_1), (1,s_2)\}$ with $s_2\geq  s_1\geq 5$. First when
$\frac{1}{s_1}+\frac{1}{s_2}<\frac{1}{3}$, $-K^3(B^{(0)})>0$. In
particular, if $s_1+s_2 \geq 13$, we see $-K^3(B)\geq -K^3(B^{(0)})\geq \frac{1}{120}$.
We are left to consider the situations: $(s_1,s_2)=(6,6)$, $(5,7)$,
$(5,6)$ and $(5,5)$.

When $(s_1,s_2)=(6,6)$, we see that $B$ is dominated by
$$B_6=\{(1,2), (2,5), (1,3), 2\times (1,6)\}$$ and thus $-K^3(B)\geq
-K^3(B_6)=\frac{1}{30}$.

When $(s_1,s_2)=(5,7)$, we see that $B$ is dominated by
$$B_7=\{(1,2), (2,5), (1,3), (1,5), (1,7)\}$$ with $-K^3(B)\geq
-K^3(B_7)=\frac{1}{42}$.

When $(s_1,s_2)=(5,6)$, we see that $B$ is dominated by one of the
following baskets $B_8$ with $-K^3(B)\geq -K^3(B_8)$:
\begin{itemize}
\item $\{2\times (2,5), (1,5), (1,6)\}$ with $-K^3=\frac{1}{30}$;

\item $\{(3,7), (1,3), (1,5), (1,6)\}$ with $-K^3=\frac{1}{70}$;

\item $\{(1,2), (3,8), (1,5), (1,6)\}$ with $-K^3=\frac{1}{120}$;

\item $\{(1,2), (2,5), (1,3), (2,11)\}$ with $-K^3=\frac{1}{330}$.
\end{itemize}

When $(s_1,s_2)=(5,5)$, because any minimal element dominated by
$B^{(0)}(B)$ has negative $-K^3$, we see that this case doesn't happen
at all.

This completes the proof for Case II.
\medskip

\noindent{\bf Case III.} $P_{-1}=2$. We have $\sigma=P_{-2} \geq
2P_{-1}-1=3$.
\medskip

{\bf Subcase III-1.} $P_{-2} \ge 5$. Since $\sigma=5$, one gets
\begin{eqnarray*}
l(3)&=&\sum_i
\{\frac{b_i(r_i-b_i)}{2r_i}+\frac{b_i(r_i-2b_i)}{r_i}\}\\
&\geq& \sum_i \frac{b_i(r_i-b_i)}{2r_i}\geq \frac{5}{4}.
\end{eqnarray*}
So $-K^3 \ge \frac{1}{2}$ by Riemann-Roch formula directly.
\medskip

{\bf Subcase III-2.} $P_{-2} =4$. If $B^{(0)}(B)=\{4\times (1,2)\}$,
then $B=B^{(0)}(B)$ and $-K^3(B)=0$ (impossible). Thus $n^0_{1,r}>0$ for some $ r \ge 3$.
Notice that $\frac{r-1}{2r} \ge \frac{1}{3}$ for $ r \ge 3$. It follows that
 $l(2) \geq
\frac{3}{4}+\frac{1}{3}=\frac{13}{12}$.  Thus we have $-K^3 \ge 2(P_{-1}-3+l(2))\geq
\frac{1}{6}$.
\medskip

{\bf Subcase III-3.} $P_{-2}=3$. We have $\sigma=3$.
Also note that $P_{-4} \ge 2 P_{-2}-1 =5$. By (4.1), we only need to
consider the situation $P_{-4}\leq 6$.

Since $n^0_{1,3}\geq 0$ and $n^0_{1,4} \ge 0$, we have
$$3P_{-3}-6 \ge P_{-4} \ge 2P_{-3}-4.$$
Thus $(P_{-3},P_{-4})=(4,5),(4,6),(5,6)$.

If $(P_{-3},P_{-4})=(5,6)$,  then $B^{(0)}(B)=\{ 3 \times (1,3)\}$ and
$B=B^{(0)}$ with $-K^3=0$, which is absurd.

If $(P_{-3},P_{-4})=(4,6)$, then by \ref{assumption}, we may assume
$\sigma_5=0$ and so that $B^{(0)}(B)=\{ (1,2), 2 \times (1,4)\}$. Again
$B=B^{(0)}(B)$ with $-K^3=0$, which is absurd.

If $(P_{-3},P_{-4})=(4,5)$, then either $B^{(0)}(B)=\{ (1,2), (1,3),
(1,4)\}$ or $B^{(0)}(B)=\{(1,2),(1,3),(1,s)\}$ with $s\geq 5$. For the
first case, any packing of $B^{(0)}$ has negative $-K^3$. Thus the
first case can not happen.

We consider the second case. If $s\geq 7$, then $$-K^3(B)\geq
-K^3(B^{(0)}(B))=\frac{1}{42}.$$ If $s\leq 6$, we consider the
one-step packing $B_1=\{(2,5), (1,s)\}$. Only when $s=6$,
$-K^3(B)\geq -K^3(B_1)=\frac{1}{30}$. This also means that $s=5$
can not happen in this situation.
\medskip

\noindent {\bf Case IV.} $P_{-1} \ge 3$. If $X$ is Gorenstein, then
$-K^3 \ge 1$ since it is an integer. Otherwise, $ -K^3 \ge 2 l(2)
\geq \frac{1}{2}$ by Riemann-Roch directly.

So we have proved the theorem. In fact, we have proved that
$-K^3=\frac{1}{330}$ if and only if $B=\{(1,2), (2,5), (1,3),
(2,11)\}$.
\end{proof}


\end{document}